\documentclass[11pt,a4paper]{amsart}

\usepackage{hyperref}

\author{Clifford Gilmore}
\title{Dynamics of Generalised Derivations and Elementary Operators}
\date{\today}

\address{Department of Mathematics and Statistics, P.O. Box 68, FI-00014 University of Helsinki, Finland.}
\thanks{This research has been supported by the Magnus Ehrnrooth Foundation and the Academy of Finland via the Centre of Excellence in Analysis and Dynamics Research (project no. 271983).}

\email{clifford.gilmore@helsinki.fi}
\keywords{Hypercyclic, generalised derivation,  elementary operator}
\subjclass[2000]{Primary 47A16; Secondary 47B47}

\usepackage[UKenglish]{babel}

\usepackage{amssymb,mathrsfs,mathtools}

\usepackage[all]{xy}

\usepackage{enumitem}

\newtheorem{thm}{Theorem}[section]

\newtheorem{cor}[thm]{Corollary}
\newtheorem{prop}[thm]{Proposition}

\theoremstyle{definition}

\newtheorem{eg}[thm]{Example}

\newtheorem{rmk}[thm]{Remark}

\numberwithin{equation}{section}		

\newcommand{\tr}[1]{\mathrm{tr}\!\left( #1 \right)} 	

\newcommand{\ptSpec}[1]{\sigma_\mathrm{p}\left( #1 \right)}

\DeclareMathOperator{\ran}{ran}

\begin{document}

\begin{abstract}
We identify concrete examples of hypercyclic generalised derivations acting  on separable ideals of operators and establish some necessary conditions for their hypercyclicity.  We also consider the dynamics of elementary operators acting on  particular Banach algebras, which reveals surprising hypercyclic behaviour on the space of bounded linear operators on the Banach space constructed by Argyros and Haydon.
\end{abstract}

\maketitle

\section{Introduction}

For a given Banach space $X$ and its space of bounded linear operators $\mathscr{L}(X)$, the \emph{generalised derivation} $\tau_{A,B} \colon \mathscr{L}(X) \to \mathscr{L}(X)$ is induced by fixed  $A, B \in \mathscr{L}(X)$ and defined as
\begin{equation*}  
\tau_{A,B}(T) = L_A(T) - R_B(T) = AT - TB
\end{equation*}
where $T \in \mathscr{L}(X)$ and $L_A, R_B \colon \mathscr{L}(X) \to \mathscr{L}(X)$ are, respectively,  the left and right multiplication operators.

Generalised derivations, also known as intertwining operators, have been studied from many aspects since the initial work by Rosenblum~\cite{Ros56}, Lumer and Rosenblum~\cite{LR59} and Anderson and Foia{\c{s}}~\cite{AF75}.  In the setting of operator ideals of $\mathscr{L}(X)$ they have been investigated by, amongst others, Fialkow~\cite{Fia78,Fia79,Fia81}, Maher~\cite{Mah92} and Kittaneh~\cite{Kit95}.  Extensive surveys on this class of operators can be found in~\cite{Fia92}, \cite{BR97} and \cite{ST06}, however their hypercyclic properties remain largely unexplored outside the special case of the commutator map $L_A - R_A$ \cite{GST16}.  

We recall for a separable Banach space $X$ that $T \in \mathscr{L}(X)$ is  \emph{hypercyclic}  if there exists  $x \in X$ such that its orbit under $T$ is dense in $X$, that is
\begin{equation*}
\overline{\{T^n x : n \geq 0\}} = X.
\end{equation*}
Linear dynamics has been a highly active area of research since the late 1980s and comprehensive accounts of the topic can be found in \cite{BM09} and \cite{GEP11}.

Hitherto in the setting of separable operator ideals, Bonet et al.~\cite{BMP04} characterised the hypercyclicity of the multipliers $L_A$ and $R_B$ and subsequently Bonilla and Grosse-Erdmann~\cite{BGE12} identified a sufficient condition for when $L_A$ is frequently hypercyclic. 
This naturally led to the study of more complicated maps built from the basic multipliers  and the investigation of the hypercyclic properties of  commutator maps $L_A - R_A$ was initiated in \cite{GST16}.  

However, hypercyclicity of  commutator maps  is quite a subtle problem and the complete answer remains unclear.  
In particular,  it was shown in  \cite{GST16} that reasonable candidates to give hypercyclic commutator maps turn out to be non-hypercyclic.

The next natural question relates to the  hypercyclicity of generalised derivations $\tau_{A,B}$ and since we are dealing in this case with pairs $(A,B)$ of operators we have more freedom.  Indeed, in Section \ref{sec:HcGD}  of this paper we  uncover concrete classes of hypercyclic generalised derivations acting  on separable operator ideals. 

In Sections \ref{sec:NonHcGD} and \ref{sec:NormalGD} we  identify many classes of non-hypercyclic generalised derivations and  we extend some observations made in \cite{GST16}  on commutator maps  to the  generalised derivation case.  
In Section \ref{sec:ElopArgHay} we prove  that certain Banach algebras do not support hypercyclic elementary operators. This  gives the surprising example of a   generalised derivation that has   different hypercyclic properties on $\mathscr{L}(X_{AH})$ and its non-trivial ideals, where $X_{AH}$ denotes the  Banach space constructed by Argyros and Haydon~\cite{AH11}.  We note   the class of elementary operators has been  studied  since the 1950s~\cite{LR59} and they have been surveyed in~\cite{Cur92}, \cite{Fia92}, \cite{ST06} and \cite{CM11}.

\section{The Setting and Hypercyclic Generalised Derivations} \label{sec:HcGD}

In this section we identify  concrete examples of hypercyclic generalised derivations and we demonstrate there is  a plentiful supply of them.  To this end we first recall  some basic definitions and establish our setting.

Hypercyclicity requires separability of the underlying space, however the space $\mathscr{L}(X)$ is usually non-separable under the operator norm topology when $X$ is a classical Banach space.  
To overcome this obstacle, one option  is to consider spaces of operators which are separable  when endowed with weaker topologies.  
Using this approach, Chan~\cite{Cha99}  investigated the hypercyclicity of the left multiplier $L_A$ on the space $\left( \mathscr{L}(H), \: SOT \right)$ under the strong operator topology, where $H$ is a separable Hilbert space.  
Further results in this setting can be found in \cite{CT01}, \cite{MRRM02}, \cite{MP03}, \cite{BMP04}, \cite{Pet07} and \cite{GM16}.

However, while it is natural to consider the question of maps acting on $\mathscr{L}(X)$ endowed with weaker topologies, our approach here is to consider generalised derivations acting on separable ideals of $\mathscr{L}(X)$.  

We say $\left(J, \lVert \: \cdot \:  \rVert_J \right)$ is a \emph{Banach ideal} of $\mathscr{L}(X)$ if
\begin{enumerate}[label=(\roman*),leftmargin=*,  itemsep=.6ex]
\item $J \subset \mathscr{L}(X)$ is a linear subspace,

\item the norm $\lVert \: \cdot \: \rVert_J$ is complete in $J$ and $\lVert S \rVert \leq  \lVert S \rVert_J$ for all $S \in J$, \label{defn:banIdealii}

\item $BSA \in J$ and $\lVert BSA \rVert_J \leq   \lVert B \rVert  \left\lVert A \right\rVert  \lVert S \rVert_J$ for $A, B \in \mathscr{L}(X)$ and $S \in J$,\label{defn:banIdealiii}

\item the rank one operators $x^* \otimes x \in J$ and $\lVert x^* \otimes x \rVert_J =   \lVert x^* \rVert   \lVert  x \rVert$ for all $x^* \in X^*$ and $x \in X$. \label{defn:banIdealiv}
\end{enumerate}
We recall that a \emph{rank one} operator  $x^* \otimes x \colon X \to X$  is defined as
\begin{equation*}
(x^* \otimes x)(z) = \langle x^*, z \rangle x = x^*(z) x
\end{equation*}
for $x^* \in X^*$, $x \in X$  and  any $z \in X$.  The space $\mathscr{F}(X)$ of \emph{finite rank} operators  is the linear span of the rank one operators, that is
\begin{equation*}
\mathscr{F}(X) =  X^* \otimes X = \left\{ \sum_{i=1}^n x_i^* \otimes x_i : x_i^* \in X^*,\;  x_i \in X \text{ and } n \geq 1 \right\}.
\end{equation*}
For brevity we say a Banach ideal $J$ is \emph{admissible} when it contains the finite rank operators as a dense subset with respect to the  norm $\lVert \: \cdot \: \rVert_J$, that is
\begin{equation*}
\overline{\mathscr{F}(X)}^{\lVert \: \cdot \: \rVert_J} = J.
\end{equation*}
When the dual $X^*$ is separable, classical examples of separable admissible Banach ideals include the space $\left( \mathscr{N}(X), \lVert \: \cdot \: \rVert_N \right)$ of nuclear operators equipped with the nuclear norm
 and the space   $\mathscr{K}(X)$ of compact operators, under the operator norm, when  $X^*$ also possesses the approximation property.
When $X$ is a separable Hilbert space, the spaces $\left( C_p, \lVert \: \cdot \: \rVert_p \right)$ of Schatten $p$-class operators  with the Schatten norm are  important examples of separable admissible Banach ideals.  We caution that  admissibility does not automatically imply separability of the ideal, for example $\left( \mathscr{N}(X), \lVert \: \cdot \: \rVert_N \right)$ is non-separable when $X^*$ is non-separable.
Further details on Banach ideals and  the approximation property can be found, for instance, in~\cite{Rya02}.

In the setting of separable operator ideals, Bonet et al.~\cite{BMP04} characterised when the mulitipliers $L_A$ and $R_B$ are hypercyclic using tensor  techniques developed by Mart{\'{\i}}nez-Gim{\'e}nez and Peris~\cite{MP03}. Their main results, expressed in the terminology of Banach ideals, are as follows: for a separable admissible Banach ideal $J \subset \mathscr{L}(X)$
\begin{enumerate}[label=\Roman*., leftmargin=*, rightmargin=1em, itemsep=1ex]
\item $L_A$ is hypercyclic on $J$ if and only if $A \in \mathscr{L}(X)$ satisfies the Hypercyclicity Criterion,

\item $R_B$ is hypercyclic on $J$ if and only if $B^* \in \mathscr{L}(X^*)$ satisfies the Hypercyclicity Criterion.
\end{enumerate}
The aforementioned Hypercyclicity Criterion is a sufficient condition for hypercyclicity which was initially demonstrated by Kitai~\cite{Kit82} and then independently rediscovered by Gethner and Shapiro~\cite{GS87}.  The standard version can be found for instance in 
\cite[Theorem 3.12]{GEP11}, however we will use the following  equivalent statement which can be found in \cite[Proposition 3.20]{GEP11}.

We say that $T \in \mathscr{L}(X)$ satisfies the \emph{Hypercyclicity Criterion} if there exists a dense linear subspace $X_0 \subset X$, an increasing sequence $(n_k)$ of positive integers and linear maps $S_{n_k}\colon X_0 \to X$, $k \geq 1$, such that for any $x \in X_0$
\begin{enumerate}[label=(\roman*), leftmargin=*, itemsep=1ex]
\item $T^{n_k} x \longrightarrow 0$,  \label{eq:hcc1}
\item $S_{n_k}x \longrightarrow 0$,  \label{eq:hcc2}
\item $T^{n_k}S_{n_k}x \longrightarrow x$ \label{eq:hcc3}	
\end{enumerate}
as $k \to \infty$. 
If $T$ satisfies the Hypercyclicity Criterion then it is hypercyclic (cf. \cite[Theorem 1.6]{BM09} or \cite[Theorem 3.12]{GEP11}).  We note, however, it was shown by de la Rosa and Read~\cite{DLRR09} that there exist hypercyclic operators that do not satisfy the criterion (cf.~\cite[Section 4.2]{BM09}).

To obtain non-trivial instances of hypercyclic generalised derivations, we naturally require that $\tau_{A,B}$ has at least dense range.   
Fialkow characterised  in \cite[Corollary 3.21]{Fia78} and \cite[Proposition 4.1]{Fia81}  when generalised derivations have dense range in the setting of certain operator ideals.  He showed the following for a separable Hilbert space $H$. (For convenience here 
$\left( C_\infty, \lVert \, \cdot \, \rVert \right)$ denotes the space of compact operators on $H$ under the operator norm topology.)

\begin{enumerate}
\item For $1 \leq p \leq \infty$, $\tau_{A,B} \colon C_p \to C_p$ has dense range if and only if $\tau_{B,A}$ is injective in $\mathscr{L}(H)$.  \label{Fialkow01}

\item $\tau_{A,B} \colon C_1 \to C_1$ is injective if and only if $\tau_{B,A} \colon C_\infty \to C_\infty$ has dense range.

\item  If $\tau_{B,A}$ has dense range in $\mathscr{L}(H)$ then $\tau_{A,B} \colon C_1 \to C_1$ is injective.

\item If $\tau_{A,B} \colon C_1 \to C_1$ has dense range then $\tau_{B,A} \colon C_\infty \to C_\infty$ is injective.
\end{enumerate}

Moreover,  Fialkow~\cite[Proposition 4.1]{Fia81} also showed  if $\tau_{A,B}$ has dense range in $\mathscr{L}(H)$ then \eqref{Fialkow01} is equivalent to $\tau_{B,A} \colon C_p \to C_p$ being injective for $1 \leq p \leq \infty$.

Obvious examples  of hypercyclic generalised derivations can be obtained by taking either $A \equiv 0$ or $B \equiv 0$ in $\tau_{A,B}$, which brings us back to the case  of the multipliers $L_A$ and $R_B$. 
We are interested here to analyse the types of hypercyclic generalised derivations that exist and our starting point  is to consider one of the most fundamental classes of hypercyclic operator: operators of the form $I + B_w$ where $I$ denotes the identity map and $B_w$ is a weighted backward shift.

Let $X = c_0$ or $\ell^p$, for $1 \leq p \leq \infty$, where $c_0$ denotes the space of sequences that tend to zero and $\ell^p$ is the space of $p$-summable sequences. The \emph{weighted backward shift} $B_w \in \mathscr{L}(X)$ is  defined as 
\begin{equation*}
B_w (x_1, x_2, x_3, \dotsc) = (w_2 x_2, w_3 x_3, \dotsc)
\end{equation*}
for $(x_n) \in X$ and where $w = \left(w_n\right)_{n \geq 2} \in \ell^\infty$ is a bounded weight sequence of nonzero scalars.  $B_w$ is bounded on $X$ if and only if $\sup_n \lvert w_n \rvert < \infty$.  If $w_n = 1$ for all $n$ we  denote the unweighted backward shift as $B \colon X \to X$.


The approach taken  in \cite{BMP04} to identify   hypercyclic left and right multipliers was to apply tensor techniques and the hypercyclic comparison principle.
While it is possible to take a similar approach here,  we prefer in Example \ref{eg:HCCgenDev} to prove directly  that the  generalised derivation satisfies the stronger property of the Hypercyclicity Criterion  on any admissible Banach ideal $J \subset \mathscr{L}(X)$  (this was not made explicit in the results of \cite{BMP04}).  Example \ref{eg:HCCgenDev} provides a basic model from which we make more general observations.
\begin{eg}  \label{eg:HCCgenDev}
Let $X = c_0$ or $\ell^p$, for $1 < p < \infty$ and  let $J \subset \mathscr{L}(X)$ be an admissible Banach ideal. That $J$ is separable follows from the separability of $X^*$. We  will demonstrate  that the generalised derivation 
\begin{equation*}
L_{B_w} - R_{-I} \colon J \to J
\end{equation*}
satisfies the Hypercyclicity Criterion, where $w = \left(w_n\right)_{n \geq 2} \in \ell^\infty$ is a bounded weight sequence of nonzero scalars such that $\sup_{n\geq 1} \lvert w_n \rvert < \infty$.

First recall that Salas~\cite{Sal95} showed the operator $I + B_w \in \mathscr{L}(X)$ is hypercyclic, while  Le{\'o}n-Saavedra and Montes-Rodr{\'{\i}}guez~\cite{LM97}  proved it  satisfies the Hypercyclicity Criterion  when  $\sup_{n\geq 1} \lvert w_n \rvert < \infty$. 

We thus assume  $\sup_n \lvert w_n \rvert < \infty$, so by \cite{LM97}  we have a dense linear subspace $X_0 \subset X$, a sequence $(n_k)$ and linear maps $S_{n_k}\colon X_0 \to X$ such that $I + B_w$ satisfies   
the Hypercyclicity Criterion.

To show that $L_{B_w} - R_{-I}$ satisfies the Hypercyclicity Criterion on $J$, we consider the subspace $X^* \otimes X_0 \subset J$, the sequence $(n_k)$ and the linear maps $L_{S_{n_k}} \colon X^* \otimes X_0 \to X^* \otimes X$ defined as
$L_{S_{n_k}} (x^* \otimes x) = x^* \otimes S_{n_k} x$
for $x^* \otimes x \in X^* \otimes X_0$.

To see that $X^* \otimes X_0$ is dense in $J$, let $x^* \otimes x \in J$.  By assumption $X_0$ is dense in $X$ and hence there exists a sequence $(x_n) \subset X_0$ that converges to $x \in X$ as $n \to \infty$. We next consider the sequence $\left( x^* \otimes x_n \right) \subset X^* \otimes X_0$ and notice that  
\begin{equation*}
\lVert x^* \otimes x_n - x^* \otimes x \lVert_J = \lVert x^* \otimes \left( x_n - x \right) \lVert_J = \left\lVert x^* \right\rVert \left\lVert x_n - x \right\rVert \longrightarrow 0.
\end{equation*}
 It then follows by linearity and admissibility that $X^* \otimes X_0$ is dense in $J$.

We note that
\begin{align*}
\left( L_{B_w} - R_{-I} \right) (x^* \otimes x) &=  x^* \otimes B_w x \; + \; x^* \otimes x  = x^* \otimes (I + B_w) x
\shortintertext{and moreover $n$-fold iteration gives}
\left( L_{B_w} - R_{-I} \right)^n (x^* \otimes x) &=   x^* \otimes (I + B_w)^n x.
\end{align*}
Next we verify that $L_{B_w} - R_{-I}$ satisfies the Hypercyclicity Criterion for any rank one operator $x^* \otimes x \in X^* \otimes X_0$.
\begin{enumerate}[label=(\roman*),leftmargin=*, rightmargin=0em, itemsep=1ex, topsep=1.5ex]
\item  
$\displaystyle \begin{aligned}[t]
\left\lVert  \left( L_{B_w} - R_{-I} \right)^{n_k} (x^* \otimes x) \right\rVert_J &= \left\lVert x^* \otimes \left( I + B_w \right)^{n_k} x \right\rVert_J \\
&= \lVert x^* \rVert \left\lVert (I + B_w)^{n_k} x \right\rVert \longrightarrow 0.
\end{aligned}$

\item $\displaystyle
\left\lVert  L_{S_{n_k}} (x^* \otimes x) \right\rVert_J = \left\lVert x^* \otimes S_{n_k} x \right\rVert_J 
= \lVert x^* \rVert \left\lVert S_{n_k} x \right\rVert \longrightarrow 0.
$

\item 
$\displaystyle \begin{aligned}[t]
\lVert \left( L_{B_w} - R_{-I} \right)^{n_k} L_{S_{n_k}}  (x^*  \otimes x &)  \; -  \; x^* \otimes x \rVert_J \\
&= \left\lVert  x^* \otimes \big( (I + B_w)^{n_k} S_{n_k}x  \; - \; x \big) \right\rVert_J \\
&= \lVert  x^* \rVert \left\lVert (I + B_w)^{n_k} S_{n_k}x \; - \; x \right\rVert \longrightarrow 0.
\end{aligned}$
\end{enumerate}
Hence  $L_{B_w} - R_{-I}$ satisfies the Hypercyclicity Criterion for rank one operators and by taking linear combinations of rank one operators 
it further follows that it is satisfied on the dense subspace  $X^* \otimes X_0 \subset J$. Thus we have obtained a hypercyclic generalised derivation on $J$.
 \hfill \qedsymbol
\end{eg}

The weighted backward shift from Example \ref{eg:HCCgenDev} is in fact a special case of the  following class of  operators.
Following the terminology of \cite[p.~219]{GEP11}, we say that $T \in \mathscr{L}(X)$ is an \emph{extended backward shift}  if
\begin{equation*}
\mathrm{span} \left( \bigcup_{j=0}^\infty \left( \ker T^j \cap \ran{T^j} \right) \right)
\end{equation*}
is dense in the Banach space $X$.

The following theorem   generalises Example \ref{eg:HCCgenDev} to extended backward shifts.  The theorem also uses the operator $e^T  \in \mathscr{L}(X)$, which is defined as
\begin{equation*}
e^T  =  \sum_{j=0}^\infty \frac{1}{j!} T^j  
\end{equation*}
for $T \in \mathscr{L}(X)$.

\begin{thm} \label{thm:extBwShiftHc}
Let $X$ be a Banach space and let $J \subset \mathscr{L}(X)$ be a separable admissible Banach ideal.  If $T \in \mathscr{L}(X)$ is an extended backward shift then 
the generalised derivations $L_T - R_{-I}$ and $L_{T'} - R_{-I}$ satisfy the Hypercyclicity Criterion on $J$, where $T' = \sum_{j=1}^\infty \frac{1}{j!} T^j$.
\end{thm}

\begin{proof}
By an unpublished result of Grivaux and Shkarin~\cite{GS07}, which can be found in
~\cite[Theorem 8.6]{GEP11}, the operators $I + T$ and $e^T$ satisfy the Hypercyclicity Criterion on $X$.   (That $e^{B_w}$ satisfies the Hypercyclicity Criterion for the weighted backward shift $B_w$ on  $c_0$ or $\ell^p$ for $1 \leq p < \infty$ was originally shown in~\cite{DSW97}.)

Applying the argument from Example \ref{eg:HCCgenDev} to the operator $I + T$, it follows that the  generalised derivation $L_T - R_{-I}$  satisfies the Hypercyclicity Criterion on $J$ and is hence hypercyclic.

Next we consider the operator $e^T \colon X \to X$.  By the result of Grivaux and Shkarin~\cite{GS07} there exists a dense linear subspace $X_0 \subset X$, an increasing sequence  of positive integers $(n_k)$ and linear maps $S_{n_k}\colon X_0 \to X$, for $k \geq 1$,  such that $e^T$ satisfies 
the Hypercyclicity Criterion.

We will   show that the Hypercyclicity Criterion is satisfied by 
\begin{equation*}
L_{T'} - R_{-I} \colon J \to J
\end{equation*} 
  for the dense subspace $X^* \otimes X_0 \subset J$, the sequence $(n_k)$ and the linear maps $L_{S_{n_k}} \colon X^* \otimes X_0 \to X^* \otimes X$ defined by $L_{S_{n_k}} (x^* \otimes x) = x^* \otimes S_{n_k}(x)$.

Notice for any rank one operator $x^* \otimes x \in X^* \otimes X_0$ that
\begin{align*}
\lVert \left( L_{T'} - R_{-I} \right)^{n_k} (x^* \otimes x) \rVert_J &= \lVert x^* \otimes \left(T' + I\right)^{n_k} x \rVert_J  \\
&= \lVert x^* \otimes \left( e^T \right)^{n_k} x \rVert_J 
= \lVert x^* \rVert \left\lVert \left( e^T \right)^{n_k} x \right\rVert \to 0
\end{align*}
so condition \ref{eq:hcc1} is satisfied.
Conditions \ref{eq:hcc2} -- \ref{eq:hcc3} follow as in Example \ref{eg:HCCgenDev} and arguing as before we have that the generalised derivation $L_{T'} - R_{-I}$ satisfies the Hypercyclicity Criterion on $J$ and is hence  hypercyclic.
\end{proof}

Another concrete example covered by Theorem \ref{thm:extBwShiftHc} is given by the extended backward shift $D_w \in \mathscr{L}(X)$ defined as
\begin{equation*}
D_w (x_1, x_2, x_3, \dotsc) = (w_2 x_2, w_4 x_4, w_6 x_6, \dotsc )
\end{equation*}
for $(x_n) \in X$ and where $X = c_0$ or $\ell^p$ for $1 < p < \infty$.
Grivaux~\cite{Gri03b}  showed that  $I + D_w \in \mathscr{L}(X)$ satisfies the Hypercyclicity Criterion, so it follows from Theorem \ref{thm:extBwShiftHc} that $I + D_w$  and $e^{D_w}$ each induce hypercyclic generalised derivations on separable admissible Banach ideals of $\mathscr{L}(X)$.

Next we introduce duality to uncover another family of  hypercyclic generalised derivations, this time induced using forward shifts.  For $X = c_0$ or $\ell^p$, for $1 \leq p \leq \infty$, the \emph{weighted forward shift}  $S_w \in \mathscr{L}(X)$ is defined as
\begin{equation*}
S_w (x_1, x_2, x_3, \dotsc) = (0, w_1 x_1, w_2 x_2, \dotsc )
\end{equation*}
for $(x_n) \in X$ and where $w = \left(w_n\right)_{n \geq 1} \in \ell^\infty$ is a bounded weight sequence of nonzero scalars.  When $w_n = 1$ for all $n$ we denote the unweighted forward shift as $S \colon X \to X$.

\begin{eg}  \label{eg:HCCgenDerRight}
Let $X = c_0$ or $\ell^p$, for $1 < p < \infty$ and  $J \subset \mathscr{L}(X)$ be an admissible Banach ideal.
We note the separability of $X^*$ implies that $J$ is separable.  
We modify the argument of Example \ref{eg:HCCgenDev} to show that 
\begin{equation*}
L_I - R_{-S_{(w_{n+1})}}  \colon J \to J
\end{equation*}
satisfies the Hypercyclicity Criterion, where $S_{(w_{n+1})} \in \mathscr{L}(X)$  is the weighted forward shift defined by
\begin{equation*}
S_{(w_{n+1})} (x_1, x_2, x_3, \dotsc) = (0, w_2 x_1, w_3 x_2, \dotsc ).
\end{equation*}  
Note that  the adjoint $(I + S_{(w_{n+1})})^* = I + B_w \in \mathscr{L}(X^*)$ satisfies the Hypercyclicity Criterion~\cite{LM97} for some dense linear subspace $X^*_0 \subset X^*$, a sequence $(n_k)$  and linear maps $S_{n_k}\colon X^*_0 \to X^*$, for $k \geq 1$.

To demonstrate that $L_I - R_{-S_{(w_{n+1})}}$  satisfies the Hypercyclicity Criterion on $J$, we consider the dense subspace $X^*_0 \otimes X \subset J$, the sequence $(n_k)$ and the linear maps $R_{S_{n_k}} \colon X^*_0 \otimes X \to X^* \otimes X$ defined as
$R_{S_{n_k}} (x^* \otimes x) = (S_{n_k}^* x^*) \otimes x$.

For any rank one operator $x^* \otimes x \in X^*_0 \otimes X$ we get that
\begin{align*}
\left\lVert \left( L_I - R_{-S_{(w_{n+1})}} \right)^{n_k} (x^* \otimes x) \right\rVert_J &= \left\lVert \left( \left( I + S_{(w_{n+1})} \right)^{* \, n_k} x^* \right) \otimes x \right\rVert_J\\
&= \left\lVert \left( \left( I + B_w \right)^{n_k} x^* \right) \otimes x \right\rVert_J \\
&= \left\lVert \left( I + B_w \right)^{n_k} x^* \right\rVert \lVert x \rVert \longrightarrow 0.
\end{align*}
Using a similar argument to Example \ref{eg:HCCgenDev}, it follows that conditions \ref{eq:hcc2} and \ref{eq:hcc3} of the Hypercyclicity Criterion are also fulfilled and hence  the generalised derivation $L_I - R_{-S_{(w_{n+1})}}$   is  hypercyclic on $J$.
\hfill \qedsymbol
\end{eg}
Example \ref{eg:HCCgenDerRight} can be extended to the following theorem similar to how Theorem \ref{thm:extBwShiftHc} generalises Example \ref{eg:HCCgenDev}.
\begin{thm} \label{thm:extBwShiftAdjHc}
Let $X$ be a Banach space and let $J \subset \mathscr{L}(X)$ be a separable admissible Banach ideal.  If $T^* \in \mathscr{L}(X^*)$ is an extended backward shift then the generalised derivations $L_I - R_{-T}$ and $L_I - R_{T'}$ satisfy the Hypercyclicity Criterion on $J$, where $T' = \sum_{j=1}^\infty \frac{1}{j!} T^j$.
\end{thm}

We note that the preceding examples of hypercyclic generalised derivations contained in Theorems \ref{thm:extBwShiftHc} and \ref{thm:extBwShiftAdjHc} are of a special and somewhat restricted type.  It would be of interest to identify instances of hypercyclic $\tau_{A,B}$ where  $A$ and $B$ are different from the identity map.  We include this problem in the list of open questions contained in Section \ref{sec:questions}.

In the next example we demonstrate that every separable admissible Banach ideal  supports hypercyclic generalised derivations.  
\begin{eg}  \label{eg:everyBanIdSupports}
Let $J \subset \mathscr{L}(X)$ be any separable admissible Banach ideal, where $X$ is an infinite dimensional separable Banach space.
We claim that $J$ supports a  generalised derivation that satisfies the Hypercyclicity Criterion.

We first recall a result of  Ansari and Bernal  that states  that $X$ supports an operator of the form $I + K \in \mathscr{L}(X)$ that satisfies the Hypercyclicity Criterion, where $I \in \mathscr{L}(X)$ is the identity operator and $K \in \mathscr{N}(X)$ is a nuclear operator (cf.~\cite[Remark 2.12]{BM09}).  Using a similar argument to Theorem \ref{thm:extBwShiftHc},  it follows that   $J \subset \mathscr{L}(X)$ admits a  generalised derivation of the form $L_K - R_{-I}$ that satisfies the Hypercyclicity Criterion.
\hfill \qedsymbol
\end{eg}

We further observe below in  Remark \ref{eg:manyHcGD} that once there exists an instance of a hypercyclic generalised derivation, then by a standard construction there exists further examples.  To do this we require the notion of quasi-factors and the hypercyclic comparison principle.
For  topological spaces $X_0$ and $X$, we say that the map $T \colon X \to X$  is a \emph{quasi-factor} of $T_0 \colon X_0 \to X_0$ if there exists a continuous map with dense range $\Psi \colon X_0 \to X$ such that $T\Psi = \Psi T_0$, that is the following diagram commutes.
\begin{equation*}
\xymatrix@C+1.5em{ 					
X_0 \ar[r]^{T_0} \ar[d]_\Psi & X_0 \ar[d]^\Psi\\
X \ar[r]^T & X
}
\end{equation*}
When $T_0$ and $T$ are linear and $\Psi$ can be taken as linear we say $T$ is a \emph{linear quasi-factor} of $T_0$.

The \emph{hypercyclic comparison principle}  states that hypercyclicity is preserved by quasi-factors and that linear quasi-factors preserve the Hypercyclicity Criterion and supercyclicity (cf.\ for instance \cite[Section 1.1.1]{BM09}).

\begin{rmk} \label{eg:manyHcGD}
Suppose  there exists a pair $(A,B)$, where $A,B \in \mathscr{L}(X)$, such that the generalised derivation $L_A - R_B$ is hypercyclic on a separable Banach ideal $J \subset \mathscr{L}(X)$.  
Then new examples of hypercyclic generalised derivations can be obtained by applying the hypercyclic comparison principle.

To see this, we consider invertible $U,V \in \mathscr{L}(X)$ and note that the generalised derivation $L_{U^{-1} A U} - R_{VBV^{-1}} \colon J \to J$  is a linear quasi-factor of $L_A - R_B$ via the following commuting diagram.
\begin{equation*}
\xymatrix@C+7em@R+1em{
J \ar[r]^{L_A - R_B} \ar[d]_{L_{U^{-1}} R_{V^{-1}}} & J \ar[d]^{L_{U^{-1}} R_{V^{-1}}}  \\
J \ar[r]^{L_{U^{-1}AU} - R_{VBV^{-1}}} & J			
}
\end{equation*}
Hence it follows by the hypercyclic comparison principle that $L_{U^{-1} A U} - R_{VBV^{-1}}$ is hypercyclic on $J$ and thus we easily obtain further instances of hypercyclic generalised derivations.
\end{rmk}

\section{Classes of Non-Hypercyclic Derivations} \label{sec:NonHcGD}

In this section we isolate large classes of operators in the opposite direction that do not induce hypercyclic generalised derivations.
We begin by identifying an elementary spectral condition which is related to the  well known fact that the adjoint of a hypercyclic operator has no eigenvalues (cf.~\cite[Proposition 1.17]{BM09}).

\begin{prop} \label{prop:ptSpecNonEmpty}
Let $X$ be a Banach space and $A, B \in \mathscr{L}(X)$.  If the point spectra $\ptSpec{A^*}$ and $\ptSpec{B}$ are both nonempty then the generalised derivation $\tau_{A,B}$ is not hypercyclic on any separable Banach ideal $J \subset \mathscr{L}(X)$.
\end{prop}

\begin{proof}
By assumption we have that $A^* x^* = \alpha x^*$ and $B x = \beta x$ for some eigenvalues $\alpha, \beta \in \mathbb{C}$  corresponding to nonzero eigenvectors $x \in X$, $x^* \in X^*$. 
We will show the adjoint  $\tau^*_{A,B} \colon J^* \to J^*$ has a nonempty point spectrum.

We define the linear functional $\varphi \in J^*$  by $\varphi(T) = \langle x^*, Tx \rangle$, where $T \in J$ and we recall that it is bounded  since for any $T \in J$
\begin{align*}
\lvert \varphi(T) \rvert = \lvert \langle x^*, Tx \rangle \rvert \leq \lVert x^* \rVert \left\lVert x \right\rVert \lVert T \rVert \leq \lVert x^* \rVert \left\lVert x \right\rVert \lVert T \rVert_J.
\end{align*}
We note for any $T \in J$ that
\begin{align*}
\langle \tau^*_{A,B} \left( \varphi \right), T \rangle &= \langle \varphi, AT - TB \rangle = \varphi(AT) - \varphi(TB) \\
&= \langle A^* x^*, Tx \rangle - \langle x^*, TBx \rangle = \alpha \langle  x^*, Tx \rangle - \beta \langle x^*,  Tx \rangle \\
&= \alpha \varphi(T) - \beta \varphi(T) = (\alpha - \beta) \langle \varphi, T \rangle.
\end{align*}
So $\alpha - \beta$ is an eigenvalue for $\tau^*_{A,B}$ 
and $\tau_{A,B}$ is not hypercyclic on $J$.
\end{proof}

We can apply Proposition~\ref{prop:ptSpecNonEmpty}, for instance, to the   generalised derivation $\tau_{S,B} \in \mathscr{L}(J)$, where $X= c_0$ or $\ell^p$ for $1 < p < \infty$, $J \subset \mathscr{L}(X)$ is a separable Banach ideal and  $B, S \in \mathscr{L}(X)$ are, respectively,  the backward  and forward shifts.
Recall that the adjoint of the forward shift is $S^* = B$ and the point spectrum of $B$ is the open unit disc $\mathbb{D}$.  Hence both $\ptSpec{B}$ and  $\ptSpec{S^*}$ are nonempty and it follows from Proposition~\ref{prop:ptSpecNonEmpty} that  $\tau_{S,B}$ is not hypercyclic on $J$.

Next we recall the notion of supercyclicity which is required in the following remark.
For a separable Banach space $X$, we say $T \in \mathscr{L}(X)$ is  \emph{supercyclic}  if there exists $x \in X$ such that its projective orbit under $T$ is dense in $X$, that is
\begin{equation*}
\overline{ \left\lbrace \lambda T^n x : n \geq 0, \: \lambda \in \mathbb{C} \right\rbrace } = X.
\end{equation*}
We note that the class of hypercyclic operators is strictly contained in the class of supercyclic operators (cf.~\cite[Example 1.15]{BM09}).  

\begin{rmk} \label{rmk:ptSpecSc}
Recall that  Herrero~\cite{Her91} proved if $T \in \mathscr{L}(X)$ is supercyclic then either $\ptSpec{T^*} = \varnothing$ or $\ptSpec{T^*} = \{ \lambda \}$ for some $\lambda \neq 0$ (cf.~\cite[Proposition 1.26]{BM09}).  Hence if either $A^*$ or $B$ from Proposition~\ref{prop:ptSpecNonEmpty} has more than one eigenvalue then $\tau_{A,B}$ is not even supercyclic on $J$.
\end{rmk}

Next we extend Proposition~\ref{prop:ptSpecNonEmpty}  to the more general class of elementary operators.
We recall for a Banach space $X$, the map  $\mathcal{E}_{A,B} \colon \mathscr{L}(X) \to \mathscr{L}(X)$ is an \emph{elementary operator}  if
\begin{equation*}
\mathcal{E}_{A,B} = \sum_{j=1}^n L_{A_{j}} R_{B_{j}}
\end{equation*}
where $A = (A_1,\ldots ,A_n)$, $B = (B_1,\ldots , B_n) \in \mathscr{L}(X)^n$ are fixed $n$-tuples given by  $A_j \in \mathscr{L}(X)$ and $B_j \in \mathscr{L}(X)$ for $j = 1, \dotsc, n$.

\begin{prop}  \label{prop:elopEvects}
Let $X$ be a Banach space and $A, B \in \mathscr{L}(X)^n$ for $n \geq 1$.  If the operators $A_j^*$ have eigenvalues sharing a common eigenvector and the operators $B_j$ have eigenvalues sharing a common eigenvector, for $1 \leq j \leq n$, then the  elementary operator $\mathcal{E}_{A,B}$ is not hypercyclic on any separable Banach ideal $J \subset \mathscr{L}(X)$.
\end{prop}

\begin{proof}
By assumption we have $A^*_j x^* = \alpha_j x^*$ and $B_j x = \beta_j x$ for eigenvalues $\alpha_j, \beta_j \in \mathbb{C}$  corresponding to nonzero eigenvectors $x \in X$ and $x^* \in X^*$, for $1 \leq j \leq n$.

As in the proof of Proposition~\ref{prop:ptSpecNonEmpty}, we define the continuous linear functional $\varphi \in J^*$ by $\varphi(T) = \langle x^*, Tx \rangle$, where $T \in J$.
Notice for the adjoint $\mathcal{E}_{A,B}^* \colon J^* \to J^*$ and any $T \in J$ that 
\begin{equation*}
\langle \mathcal{E}_{A,B}^* \left( \varphi \right), T \rangle =  \sum_{j=1}^n \varphi\left( A_j T B_j \right) =  \sum_{j=1}^n \langle A_j^* x^*,  T B_j x\rangle = \sum_{j=1}^n \alpha_j \beta_j \langle \varphi,  T \rangle.
\end{equation*}
Hence $\sum_{j=1}^n \alpha_j \beta_j$ is an eigenvalue of $\mathcal{E}^*_{A,B}$ 
 and $\mathcal{E}_{A,B}$ is not hypercyclic on $J$.
\end{proof}

To illustrate Proposition~\ref{prop:elopEvects} let $X= c_0$ or $\ell^p$ for $1 < p < \infty$.
We consider the separable Banach ideal $J \subset \mathscr{L}(X)$ and the elementary operator  
\begin{equation*}
\mathcal{E}_{U,V} \colon J \to J
\end{equation*}
where $U = (S, I, S^2)$, $V = (I, B, B^2) \in \mathscr{L}(X)^3$ and $B, S \in \mathscr{L}(X)$ are the  backward  and forward shift operators.
A common eigenvector for the 3-tuple $V$ is easily obtained,  we choose $\beta \in \mathbb{D}$ which gives $x = (1, \beta, \beta^2, \dotsc) \in X$ such that $Bx = \beta x$.  Hence $1, \beta,\beta^2$ are, respectively, eigenvalues for $I, B, B^2$.  Similarly we choose $\alpha \in \mathbb{D}$ which gives the eigenvector $x^* = (1, \alpha, \alpha^2, \dotsc) \in X^*$ such that $\alpha, 1, \alpha^2$ are the corresponding eigenvalues for $S^*, I, S^{* \, 2}$.  So it follows from Proposition~\ref{prop:elopEvects} that $\mathcal{E}_{U,V}$ is not hypercyclic on $J$.

Next we extend some observations made in \cite{GST16} concerning commutator maps  induced by Riesz operators   to the generalised derivation case.  
We recall that $T \in \mathscr{L}(X)$ is a \emph{Riesz} operator if its essential spectrum $\sigma_{\mathrm{e}}(T) = \{ 0 \}$ and they are never hypercyclic~\cite[p.\ 160]{GEP11}.
The spectrum of $T$ is $\sigma(T) = \{ 0\} \cup \{ \lambda_n : n \geq 1 \}$, where $\{ \lambda_n : n \geq 1 \}$ is a discrete, at most countable (possibly empty) set containing nonzero eigenvalues of finite multiplicity. 

The spectrum of the generalised derivation $\tau_{A,B}$  acting on $\mathscr{L}(X)$ was initially identified by Lumer and Rosenblum~\cite{LR59}. 
Their formula extends to $\tau_{A,B}$ restricted to any Banach ideal $J \subset \mathscr{L}(X)$ and
satisfies 
\begin{equation*}     
\sigma_J(\tau_{A,B}) = \sigma(A) - \sigma(B)
\end{equation*}
where $\sigma_J(\tau_{A,B})$ denotes the spectrum of $\tau_{A,B} \colon J \to J$.
Further details can be found, for instance, in the survey~\cite[Theorem 3.12]{ST06}.

We also recall Kitai's~\cite{Kit82} spectral condition that every connected component of the spectrum of a hypercyclic operator intersects the unit circle (cf.~\cite[Theorem 1.18]{BM09}).
\begin{thm}  \label{thm:derivRieszNotHc}
Let $X$ be a Banach space and $J \subset \mathscr{L}(X)$ be a separable Banach ideal.  If $A$, $B \in \mathscr{L}(X)$ are Riesz operators then the induced generalised derivation $\tau_{A,B} \colon J \to J$
is not hypercyclic.
\end{thm}

\begin{proof}
The spectrum of $\tau_{A,B}$ on $J$ is given by
\begin{align*}
\sigma_J(\tau_{A,B}) &= \sigma(A) - \sigma(B) \\
&= \left\lbrace \alpha_m - \beta_n :   \alpha_m \in \sigma(A),\, \beta_n \in \sigma(B),\, m,n \geq 0 \right\rbrace
\end{align*}
where for convenience we set $\alpha_0 = 0 = \beta_0$.  

Note that $\sigma_J(\tau_{A,B})$ is a closed and compact set which is at most countable.
So it is a discrete set containing the singleton $\{ 0 \}$ as a connected component and it follows from the spectral condition of Kitai that $\tau_{A,B}$ is not hypercyclic on $J$.
\end{proof}
We remark that   compact operators are an important class of Riesz operators and hence it follows by Theorem~\ref{thm:derivRieszNotHc} that $\tau_{A,B}$ is not hypercyclic on any separable Banach ideal $J \subset \mathscr{L}(X)$ for compact $A, B \in \mathscr{K}(X)$.

\section{Normal Derivations} \label{sec:NormalGD}

The Hilbert space setting provides interesting  families of non-hypercyclic operators and in this section we extend some observations made in \cite{GST16} to the generalised derivation case.
Here $H$ denotes an infinite dimensional Hilbert space over the complex field.

Generalised derivations     induced by normal operators and  acting on $\mathscr{L}(H)$, or \emph{normal derivations},  
 were first studied by Anderson~\cite{And73b} and Anderson and Foia{\c{s}}~\cite{AF75}.  
The Banach ideal case was then investigated by Maher~\cite{Mah92} and Kittaneh~\cite{Kit95}.

We recall that $T \in \mathscr{L}(H)$ is  \emph{positive}, denoted $T \geq 0$, if the inner product $\langle Tx , x \rangle \geq 0$ for all $x \in H$.
 We say $T \in \mathscr{L}(H)$ is \emph{hyponormal} if $T^* T - T T^* \geq 0$  or equivalently if $\left\lVert Tx  \right\rVert \geq \left\lVert T^* x \right\rVert$ for all $x \in H$.
The class of hyponormal operators contains some well known classes of operators such as the subnormal, normal and self-adjoint operators~\cite{Hal82}. 
Kitai~\cite{Kit82} showed hyponormal operators are never hypercyclic 
and Bourdon~\cite{Bou97} proved that they are never even supercyclic.

We consider generalised derivations induced by hyponormal operators acting on the space of Hilbert-Schmidt operators $C_2$, which is complete in the Hilbert-Schmidt norm 
\begin{equation*}
\lVert T \rVert^2_2 = \tr{T^*T} = \sum_j \langle T^* T e_j, e_j \rangle = \sum_j \langle Te_j, Te_j \rangle = \sum_j \lVert Te_j \rVert^2
\end{equation*}
where $T \in C_2$, $\tr{T}$ denotes the trace of $T$ and $\left( e_j \right)$ is any orthonormal basis of $H$.  It is a Hilbert space with the corresponding inner product 
\begin{equation*}
\langle S, T \rangle = \tr{T^* S}
\end{equation*}
and further details on Hilbert-Schmidt operators can be found in \cite{Con00}.

We note that a characterisation of hyponormal generalised derivations is stated, without proof, in \cite[p.\ 50]{Con91}.  In the proof of Theorem~\ref{thm:hnNeverHc} below we  recall the part we need for the convenience of the reader.
\begin{thm}  \label{thm:hnNeverHc}
Let $A, B \in \mathscr{L}(H)$ be such that $A$ and $B^*$ are hyponormal.  Then the generalised derivation $\tau_{A,B} \colon C_2 \to C_2$ is not supercyclic.
\end{thm}

\begin{proof}
We first observe that 
\begin{align*}
\tau_{A,B}^* \tau_{A,B} - \tau_{A,B} \tau_{A,B}^* &= \left( L_{A^*} - R_{B^*} \right) \left( L_A - R_B \right) - \left( L_A - R_B \right) \left( L_{A^*} - R_{B^*} \right) \\
&= L_{A^* A} - L_{A^*} R_B - R_{B^*} L_A + R_{B B^*}  \nonumber \\
& \qquad {} - L_{A A^*} + L_A R_{B^*} + R_B L_{A^*} - R_{B^* B} \label{line:hypoComm}\\
&= L_{A^*A - AA^*} + R_{BB^* - B^*B}
\end{align*}
where above we have used the facts that $L_S R_T = R_T L_S$ and $R_S R_T = R_{TS}$ for any $S, T \in \mathscr{L}(H)$.

For any $T \in C_2$ notice that
\begin{align*}
\left\langle L_{A^*A - AA^*}\left( T \right), T \right\rangle = \left\langle A^*AT - AA^*T, T \right\rangle &= \left\langle AT, AT \right\rangle - \left\langle A^*T, A^*T \right\rangle \\
&= \left\lVert AT \right\rVert_2^2 - \lVert A^*T \rVert^2_2.
\end{align*}
Since $A$ is hyponormal we know that $\lVert ATe_n \rVert \geq \left\lVert A^*Te_n \right\rVert$ for all $n \geq 1$ and any orthonormal basis $\left( e_n \right)$ of $H$.  Hence
\begin{equation*}
\lVert AT \rVert_2^2 = \sum_n \left\lVert ATe_n \right\rVert^2 \geq \sum_n   \left\lVert A^*Te_n \right\rVert^2 = \left\lVert A^*T \right\rVert^2_2
\end{equation*}
and hence $\left\lVert AT \right\rVert_2^2 - \left\lVert A^*T \right\rVert^2_2 \geq 0$ for all $T \in C_2$.

Similarly for any $T \in C_2$
\begin{equation*}
\left\langle R_{BB^* - B^*B} \left( T \right), T \right\rangle = \left\langle B^*T , B^*T \right\rangle -  \left\langle BT, BT \right\rangle = \left\lVert B^*T \right\rVert_2^2 - \left\lVert BT \right\rVert^2_2
\end{equation*}
and  it follows from the hyponormality of $B^*$ that $\lVert B^*T \rVert_2^2 - \lVert BT \rVert^2_2 \geq 0$.

Hence $\tau_{A,B}^* \tau_{A,B} - \tau_{A,B} \tau_{A,B}^* \geq 0$ so $\tau_{A,B}$ is hyponormal and cannot be supercyclic on $C_2$ by \cite{Bou97}.
\end{proof}

Next we extend Theorem~\ref{thm:hnNeverHc}   using the hypercyclic comparison principle.
\begin{cor}
Let $A,B \in  \mathscr{L}(H)$ be such that $A$ and $B^*$ are hyponormal and let $J$ be an admissible Banach ideal contained in $C_2$.  Then the induced generalised derivation  $\tau_{A,B}$ is never supercyclic on $J$.
\end{cor}

\begin{proof}
The finite rank operators are contained in $J$ and they form a dense subset of $C_2$. This gives a natural linear inclusion, with dense range, $\Psi \colon J \rightarrow C_2$.
Hence $\tau_{A,B} \colon C_2 \to C_2$ is a linear quasi-factor of $\tau_{A,B} \colon J \to J$ via the following commuting diagram.
\begin{equation*}
\xymatrix@C+2em{ 
J \ar[r]^{\tau_{A,B}} \ar[d]_\Psi & J \ar[d]^\Psi\\
C_2 \ar[r]^{\tau_{A,B}} & C_2
}
\end{equation*}  
If $\tau_{A,B}$ was supercyclic on $J$ then it would follow by the comparison principle that it is supercyclic on $C_2$.  However we know from Theorem~\ref{thm:hnNeverHc} that $\tau_{A,B}$ is not supercyclic on $C_2$ and hence $\tau_{A,B}$ is not supercyclic on $J$.
\end{proof}

\section{Elementary Operators on the Argyros-Haydon Space} \label{sec:ElopArgHay}

Argyros and Haydon~\cite{AH11} resolved the famous \emph{scalar-plus-compact} problem with the construction of the extreme  Banach space $X_{AH}$.  While this type of  space is rare, $X_{AH}$ is relatively nice and possesses many remarkable properties.  In this section we reveal some surprising differences in the hypercyclic behaviour of elementary operators acting on  $\mathscr{L}(X_{AH})$ and $\mathscr{K}(X_{AH})$.

We first recall some   properties of $X_{AH}$ relevant to our discussion.  The space $X_{AH}$ has a Schauder basis and every $T \in \mathscr{L}(X_{AH})$ is of the form
\begin{equation}  \label{eq:operFormAH}
T = \lambda I + K
\end{equation}
where $\lambda \in \mathbb{C}$ and $K \in \mathscr{K}(X_{AH})$ is a compact operator.

The existence of a Schauder basis implies that $X_{AH}$ possesses the approximation property and it is  shown in~\cite{AH11} that the dual  $X_{AH}^*$ is isomorphic to the sequence space $\ell^1$ and  is therefore separable.
Hence the space of compact operators $\mathscr{K}(X_{AH})$ is a separable admissible Banach ideal under the operator norm topology and  it further follows that 
$\mathscr{L}(X_{AH}) = \mathbb{C} \cdot I + \mathscr{K}(X_{AH})$ is  separable under the operator norm topology.

The separability of $\mathscr{L}(X_{AH})$ naturally leads  to the question of whether it supports hypercyclic elementary operators.  
Using an argument similar to Saldivia~\cite[Proposition 5.8]{Sal03} and partially outlined in~\cite[p.~300]{GEP11} for the multiplier case, we establish a somewhat more general observation for  elementary operators acting on particular Banach algebras and we then apply it to  $\mathscr{L}(X_{AH})$.

For a Banach algebra $\mathscr{A}$, the elementary operator $\mathcal{E}_{a,b} \colon \mathscr{A} \to \mathscr{A}$ is given by 
\begin{equation*} 
\mathcal{E}_{a,b} = \sum_{j=1}^n L_{a_j} R_{b_j}
\end{equation*}
where $a = (a_1,\ldots ,a_n)$, $b = (b_1,\ldots , b_n) \in \mathscr{A}^n$, $n \geq 1$ and we define $L_{a_j}(s) = a_js$, $R_{b_j}(s) = sb_j$ for any $s \in \mathscr{A}$ and $j=1, \dotsc, n$.

We further recall that a \emph{multiplicative linear functional} $\varphi \colon \mathscr{A} \to \mathbb{C}$ is a nonzero linear functional such that $\varphi(ab) = \varphi(a)\varphi(b)$ for all $a, b \in \mathscr{A}$ and it is well known that they are always continuous.

\begin{thm} \label{thm:ElopNotHcBanAlgMulFnal}
Let $\mathscr{A}$ be a Banach algebra which admits a non-trivial multiplicative linear functional $\varphi \in \mathscr{A}^*$. Then the elementary operator $\mathcal{E}_{a,b} \colon \mathscr{A} \to \mathscr{A}$ is not hypercyclic.
\end{thm}

\begin{proof}
We will prove that $\mathcal{E}_{a,b}$ is not hypercyclic on $\mathscr{A}$ by showing that $\varphi \in \mathscr{A}^*$ is an eigenvector of the adjoint $\mathcal{E}_{a,b}^* \colon \mathscr{A}^* \to \mathscr{A}^*$.

Notice for any $s \in  \mathscr{A}$  that
\begin{align}
\langle \mathcal{E}_{a,b}^*(\varphi) , \: s  \rangle  &= \langle \varphi , \: \mathcal{E}_{a,b} (s) \rangle = \varphi\bigg( \sum_{j=1}^n a_j s b_j \bigg) \nonumber \\
&=  \sum_{j=1}^n \varphi(a_j) \varphi(s) \varphi(b_j)  = \varphi(s) \sum_{j=1}^n \varphi(a_j) \varphi(b_j) \label{eq:PhiLinMultip} \\
&=  \bigg( \sum_{j=1}^n \varphi(a_j) \varphi(b_j) \bigg) \langle \varphi , \: s \rangle  \nonumber 	
\end{align}
where \eqref{eq:PhiLinMultip} follows from the linearity and multiplicativity of $\varphi$.
Hence $\varphi$ is an eigenvector of  $\mathcal{E}_{a,b}^*$ corresponding to the eigenvalue $\sum_{j=1}^n \varphi(a_j) \varphi(b_j)$ and it follows that $\mathcal{E}_{a,b}$ is not hypercyclic on $\mathscr{A}$.
\end{proof}

When the Banach algebra $\mathscr{A}$ contains a unit element we note that  Saldivia~\cite[Theorem 5.3]{Sal03}  showed the left and right multipliers are not topologically transitive on $\mathscr{A}$.  
We recall that $T \colon \mathscr{A} \to \mathscr{A}$ is  a \emph{left multiplier} if T(ab) = T(a)b and a \emph{right multiplier} if $T(ab) =aT(b)$ for all $a, b \in \mathscr{A}$.   
When $\mathscr{A}$ is an infinite dimensional and separable Banach algebra the Birkhoff Transitivity Theorem (cf.~\cite[Theorem 1.2]{BM09}) implies that topological transitivity is equivalent to hypercyclicity. So when $\mathscr{A}$ contains a unit element  this gives an alternative proof that  $L_a$ and $R_a$ are not hypercyclic on $\mathscr{A}$ for any $a \in  \mathscr{A}$.

Next we apply Theorem~\ref{thm:ElopNotHcBanAlgMulFnal} to the Banach algebra $\mathscr{L}(X_{AH})$.
\begin{cor} \label{cor:ElopNotHcArgHay}
No elementary operator is hypercyclic on $\mathscr{L}(X_{AH})$.
\end{cor}

\begin{proof}
We define the linear functional $\varphi \colon \mathscr{L}(X_{AH}) \to \mathbb{C}$ by
\begin{equation*}
\varphi(\lambda I + K) = \lambda
\end{equation*}
where $\lambda \in \mathbb{C}$ and $K \in \mathscr{K}(X_{AH})$.
This is well defined since it follows from \eqref{eq:operFormAH} that the representation  of any operator in $\mathscr{L}(X_{AH})$ is unique.  

Further notice $\varphi$ is a multiplicative functional since for any $T = \lambda_0 I + K_0$ and $S = \lambda I + K \in \mathscr{L}(X_{AH})$ we have that 
\begin{align*}
\varphi\left(TS\right) &=  \varphi \left( \lambda_0 \lambda I + \lambda_0 K + \lambda K_0 + K_0 K \right) = \lambda_0 \lambda
\shortintertext{and}
\varphi\left(T\right)\varphi\left(S\right) &=  \varphi \left( \lambda_0 I + K_0 \right) \varphi \left( \lambda I + K \right) =  \lambda_0 \lambda.
\end{align*}
So it follows by Theorem~\ref{thm:ElopNotHcBanAlgMulFnal} that elementary operators are not hypercyclic on $\mathscr{L}(X_{AH})$.
\end{proof}

Another example covered by Theorem \ref{thm:ElopNotHcBanAlgMulFnal} is provided by the  Banach spaces $\mathfrak{X}_k$, which were constructed by Tarbard~\cite{Tar12} for any given $k \geq 2$.   In \cite{Tar12} he showed that the dual $\mathfrak{X}_k^*= \ell_1$, that the space $\mathscr{K}(\mathfrak{X}_k)$ of compact operators on $\mathfrak{X}_k$ is separable and  that $\mathfrak{X}_k$ admits a non-compact, strictly singular $S \in \mathscr{L}(\mathfrak{X}_k)$, with $S^j \neq 0$ for $1 \leq j < k$ and $S^k = 0$, such that every $T \in \mathscr{L}(\mathfrak{X}_k)$ can be uniquely represented as
\begin{equation}  \label{defn:TarbardOper}
T = \sum_{j=0}^{k-1} \lambda_j S^j + K
\end{equation}
where $\lambda_j  \in \mathbb{R}$ for $0 \leq j \leq k-1$ and  $K \in \mathscr{K}(\mathfrak{X}_k)$.  
Since $\mathscr{K}(\mathfrak{X}_k)$ has finite codimension in $\mathscr{L}(\mathfrak{X}_k)$  it follows that $\mathscr{L}(\mathfrak{X}_k)$ is separable under the operator norm topology.

It is not difficult to check that $\mathscr{L}(\mathfrak{X}_k)$ supports the non-trivial,  multiplicative linear functional $\varphi$ defined by $\varphi(T) = \lambda_0$, where $T \in \mathscr{L}(\mathfrak{X}_k)$ is as given in \eqref{defn:TarbardOper}.
So it follows by  Theorem \ref{thm:ElopNotHcBanAlgMulFnal} that no elementary operator is hypercyclic on $\mathscr{L}(\mathfrak{X}_k)$.

We note the argument from Corollary~\ref{cor:ElopNotHcArgHay} does not apply to $\mathscr{K}(X_{AH})$ since $\varphi \big|_{\mathscr{K}(X_{AH})} \equiv 0$. In fact, in the next example we see that $\mathscr{K}(X_{AH})$  supports hypercyclic generalised derivations, which  further reveals the  subtle nature of the dynamical behaviour of elementary operators. 

\begin{eg} \label{eg:LmulHcCptArgHay}
Since $\mathscr{K}(X_{AH})$ is an admissible Banach ideal, we recall from Example \ref{eg:everyBanIdSupports} that it supports a generalised derivation that satisfies the Hypercyclicity Criterion.

In particular, \cite[Theorem 8.9]{GEP11} gives that there exists some $T  = \lambda I + K \in \mathscr{L}(X_{AH})$, for $\lambda \in \mathbb{C}$ and $K \in \mathscr{K}(X_{AH})$, such that $T$ satisfies the Hypercyclicity Criterion. 
So it follows that the generalised derivation $L_K - R_{-\lambda I} \colon \mathscr{K}(X_{AH}) \to \mathscr{K}(X_{AH})$ satisfies the Hypercyclicity Criterion.
\hfill \qedsymbol
\end{eg}

Hence the seemingly minor difference of  one dimension between the spaces $\mathscr{K}(X_{AH})$ and $\mathscr{L}(X_{AH})$ completely alters the hypercyclicity of the generalised derivation $L_K - R_{-\lambda I}$.  
The delicate nature of this question  is further illustrated below where we show that $\mathscr{K}(X_{AH})$ does not support any hypercyclic commutator maps.

\begin{prop}   \label{prop:CommNotHcCptArgHay}
Let  $A \in \mathscr{L}(X_{AH})$.  Then the commutator map $\Delta_A = L_A - R_A$ is not hypercyclic on $\mathscr{K}(X_{AH})$.
\end{prop}

\begin{proof}
The operator $A \in \mathscr{L}(X_{AH})$ has the form $A = \lambda I + K$ for some $\lambda \in \mathbb{C}$ and $K \in \mathscr{K}(X_{AH})$.  So $\Delta_A \colon \mathscr{K}(X_{AH}) \to \mathscr{K}(X_{AH})$  is given by
\begin{equation*}
\Delta_A = L_{\lambda I + K} - R_{\lambda I + K} =  L_K - R_K = \Delta_K.
\end{equation*}
Since $\mathscr{K}(X_{AH})$ is a separable Banach ideal and $K$ is compact it follows from Theorem~\ref{thm:derivRieszNotHc} that $\Delta_A$ is not hypercyclic on $\mathscr{K}(X_{AH})$.
\end{proof}
The argument from Proposition~\ref{prop:CommNotHcCptArgHay} can also be used to prove directly that no commutator map is hypercyclic on $\mathscr{L}(X_{AH})$.

The next example reveals another twist in the dynamical behaviour of  elementary operators.  We will see they act quite differently on the space $\left( \mathscr{L}(X_{AH}),\: SOT \right)$ endowed with the strong operator topology.

\begin{eg} \label{eg:SOTbehaviour}
We claim that there exists a left multiplier $L_T$ that displays contrasting hypercyclic behaviour on the space $\left( \mathscr{L}(X_{AH}),\: SOT \right)$ and the space $\left( \mathscr{L}(X_{AH}), \: \lVert \, \cdot \, \rVert \right)$ endowed with the operator norm topology.

In fact, combining previously known results, we recall that \cite[Theorem 8.9]{GEP11} 	gives that there exists an operator $T \in \mathscr{L}(X_{AH})$ that satisfies the Hypercyclicity Criterion and by \cite[Corollary 6]{CT01} it further follows that $L_T$ is hypercyclic on $\left( \mathscr{L}(X_{AH}), \: SOT \right)$.  
This contrasts with the behaviour described in Corollary \ref{cor:ElopNotHcArgHay}, from which it follows that $L_T$ is not hypercyclic on $\left( \mathscr{L}(X_{AH}), \: \lVert \, \cdot \, \rVert \right)$.
\hfill \qedsymbol
\end{eg}

To explain the diverging behaviour from Example \ref{eg:SOTbehaviour}, we note that the space $\mathscr{F}(X_{AH})$ of finite rank operators on $X_{AH}$ is dense in  $\left( \mathscr{L}(X_{AH}), \: SOT \right)$.  The density of $\mathscr{F}(X_{AH})$  in $\left( \mathscr{K}(X_{AH}), \: \lVert \, \cdot \, \rVert \right)$ is also used in the  results which yield instances of hypercyclic elementary operators in this setting.
However, it does not hold that $\mathscr{F}(X_{AH})$  is dense in  $\left( \mathscr{L}(X_{AH}), \: \lVert \, \cdot \, \rVert \right)$ and  we thus arrive at the situation described by Corollary \ref{cor:ElopNotHcArgHay}.

\section{Further Questions}  \label{sec:questions}

Some natural questions arising from this paper include the following.
\begin{enumerate}[label=\arabic*., leftmargin=*, rightmargin=1em, itemsep=1ex] 
\item Do  reasonable sufficient conditions exist on the pair $(A, B)$ that induce hypercyclic $\tau_{A,B}$ on separable Banach ideals?

\item Corollary~\ref{cor:ElopNotHcArgHay} and Example~\ref{eg:LmulHcCptArgHay} contrast the hypercyclic properties of particular generalised derivations on different spaces.  Furthermore in \cite{GST16} it is shown for particular Banach spaces $X$ that $\mathscr{N}(X)$ does not support a hypercyclic commutator map.  Do there exist further examples of contrasting hypercyclic behaviour of elementary operators on  different Banach ideals?
\end{enumerate}

\section*{Acknowledgement}

This article is part of the PhD thesis of the author and he would like to thank his supervisor Hans-Olav Tylli for helpful comments and remarks during its preparation.

\bibliographystyle{abbrv}

\end{document}